\documentclass[a4paper,11pt]{article}
\usepackage[top=1in, bottom=1in, left=0.9in, right=0.9in]{geometry}
\usepackage{amsmath,amssymb,amsthm}
\usepackage{mathdots}
\usepackage{caption}
\usepackage{booktabs}
\usepackage{longtable}
\usepackage{float}
\usepackage{tikz-cd}

\usepackage[T1]{fontenc}
\usepackage[type1]{libertine}
\usepackage[libertine]{newtxmath}

\usepackage[shortlabels]{enumitem}
\usepackage{hyperref}

\usepackage{makeidx}
\usepackage{boxedminipage}
\usepackage{epic}
\usepackage{longtable}
\usepackage{ifthen}
\usepackage{tabularx}
\usepackage{verbatim}
\usepackage{tikz}

\newtheorem{theorem}{Theorem}[section]
\newtheorem*{theorem-no-number}{Theorem}
\newtheorem{proposition}[theorem]{Proposition}
\newtheorem{lemma}[theorem]{Lemma}
\newtheorem{fact}[theorem]{Fact}
\newtheorem{problem}[theorem]{Problem}
\newtheorem{question}[theorem]{Question}
\newtheorem{conjecture}[theorem]{Conjecture}

\theoremstyle{definition}
\newtheorem{definition}[theorem]{Definition}

\theoremstyle{remark}
\newtheorem{remark}[theorem]{Remark}
\newtheorem{example}[theorem]{Example}

\newcommand{\N}{\mathbb{N}}
\newcommand{\Z}{\mathbb{Z}}
\newcommand{\R}{\mathbb{R}}
\newcommand{\C}{\mathbb{C}}

\DeclareMathOperator{\OO}{O}
\DeclareMathOperator{\PSL}{PSL}
\DeclareMathOperator{\SL}{SL}
\DeclareMathOperator{\SO}{SO}
\DeclareMathOperator{\SU}{SU}
\DeclareMathOperator{\UU}{U}

\newcommand{\ab}{\mathfrak{a}}
\newcommand{\g}{\mathfrak{g}}
\newcommand{\h}{\mathfrak{h}}
\newcommand{\kk}{\mathfrak{k}}
\newcommand{\lfrak}{\mathfrak{l}}
\newcommand{\p}{\mathfrak{p}}

\newcommand{\oo}{\mathfrak{o}}
\newcommand{\ssl}{\mathfrak{sl}}

\newcommand{\Hyp}{\mathcal{H}}
\newcommand{\Sym}{\mathfrak{S}}

\DeclareMathOperator{\Ad}{Ad}
\DeclareMathOperator{\ad}{ad}
\DeclareMathOperator{\diag}{diag}
\DeclareMathOperator{\Hom}{Hom}

\DeclareMathOperator{\Stab}{Stab}
\DeclareMathOperator{\rank}{rank}
\DeclareMathOperator{\vcd}{vcd}

\newcommand{\br}[1]{\lbrack #1 \rbrack}
\newcommand{\ep}{\varepsilon}

\title{Exotic proper actions on homogeneous spaces via convex cocompact representations}
\author{Maciej Boche\'nski and Yosuke Morita}
\date{}

\begin{document}

\maketitle

\begin{abstract}
We construct a series of homogeneous spaces $G/H$ of reductive type 
which admit proper actions of discrete subgroups of $G$ isomorphic to cocompact lattices of $\OO(n,1)$ ($n=2,3,4$) 
but do not admit proper actions of non-compact semisimple subgroups of $G$. 
The existence of such homogeneous spaces was previously not known even for $n=2$. 
Our construction of proper actions of discrete subgroups is based on 
Gu\'eritaud--Kassel's work on convex cocompact subgroups of $\OO(n,1)$ and 
Danciger--Gu\'eritaud--Kassel's work on right-angled Coxeter groups. 
On the other hand, the non-existence of proper actions of non-compact semisimple subgroups is proved by 
the theory of nilpotent orbits and elementary combinatorics.
\end{abstract}

\section{Introduction}

\subsection{The problem and the main result}

Let $G$ be a linear reductive Lie group and $H \subset G$ a reductive subgroup. 
If a discrete subgroup $\Gamma \subset G$ acts properly on the homogeneous space $G/H$, 
then the double coset space $\Gamma \backslash G/H$ is called a \textit{Clifford-Klein form}. 
We also say that $\Gamma$ is a \emph{discontinuous group} for $G/H$. 
We are interested in the non-Riemannian case, i.e.\ the case where $H$ is non-compact. 
In this situation, not every discrete subgroup of $G$ acts properly on $G/H$. 
Indeed, lattices of $G$ never act properly on $G/H$ by, for instance, Moore's ergodicity theorem. 

Fortunately, there is a convenient way of obtaining non-trivial examples of Clifford--Klein forms: 
if a reductive non-compact subgroup $L$ of $G$ acts properly on $G/H$, then any closed subgroup of $L$ 
(and in particular, any discrete subgroup of $L$) also acts properly on $G/H$. 
Proper actions of discrete subgroups obtained in this way, 
as well as the associated Clifford--Klein forms, are called \emph{standard}. 
For some $G/H$, this construction even yields \emph{cocompact} proper actions: 

\begin{example}[{Kobayashi~\cite[Prop.~4.9]{Kob89}; see also Kulkarni~\cite[Thm.~6.1]{Kul81}}]
Let $G = \OO(2n,2)$, $H = \UU(n,1)$, and $L = \OO(2n,1)$. 
Then, we have $G = LH$, and furthermore, $H \cap L = \UU(n)$ is compact. 
Thus, $L$ acts properly and transitively on $L / (H \cap L) \cong G/H$. 
Taking $\Gamma$ to be a cocompact lattice of $L = \OO(2n,1)$, 
we obtain a compact Clifford--Klein form of $G/H = \OO(2n,2) / \UU(n,1)$. 
\end{example}

Of course, not all proper actions are standard. 
In fact, it is now known that some compact Clifford--Klein forms are non-standard in a very strong sense:

\begin{example}[{Lee--Marquis~\cite[Thm.~A]{LM19} + Gu\'eritaud--Guichard--Kassel--Wienhard~\cite[Cor.~1.9]{GGKW17} for $n=2,3,4$; Monclair--Schlenker--Tholozan~\cite[Cor.~1.9]{MST23+} for general $n$}]\label{ex:non-lattice}
For any $n \geqslant 2$, there exists a compact Clifford--Klein form of $\OO(2n,2) / \UU(n,1)$
whose discontinuous group $\Gamma$ is not commensurable to any lattice of any linear reductive Lie group.
\end{example}

In general, a proper action which is not even a deformation of a standard proper action is called \emph{exotic}
(see Salein~\cite{Sal00} and Lakeland--Leininger~\cite{LL17} 
for other examples of exotic compact Clifford--Klein forms).

One of the main themes in the theory of Clifford--Klein forms is the following general question, 
with varying meanings of the word `large': 

\begin{center}
\emph{To what extent the existence of a `large' discontinuous group $\Gamma$ for $G/H$ implies 
the existence of a `large' reductive subgroup $L$ acting properly on $G/H$?}
\end{center}

One instance of this question, where `large' means `cocompact', is Toshiyuki Kobayashi's conjecture which asserts that 
\emph{a homogeneous space of reductive type which admits a compact Clifford--Klein form 
also admits a standard compact Clifford--Klein form} 
(\cite[\S 6, (vi)]{Kob97}, \cite[Conj.~4.3]{Kob01}).
It is recognized as one of the biggest open problems in this area, and its solution seems to be still out of reach. 
Note that the conjecture does \emph{not} say that every compact Clifford--Klein form is standard. 
In particular, Example~\ref{ex:non-lattice} is not a counterexample to this conjecture. 

Another important topic in this area to study the above question for a \emph{fixed} class of discrete groups $\Gamma$. 
In this paper, we study this problem in the case where 
$\Gamma$ is isomorphic to a cocompact lattice of $\OO(n,1)$ for $n = 2,3$, and $4$.  
A simplified version of our main result is the following: 

\begin{theorem}[cf.\ Theorem~\ref{thm:main-example}]\label{thm:main-simplified}
There is a homogeneous space $G/H$ of reductive type with the following two properties: 
\begin{itemize}
\item For $n = 2,3$, and $4$, there exists a discrete subgroup of $G$ 
which is isomorphic to a cocompact lattice of $\OO(n,1)$ and acts properly on $G/H$. 
\item There does not exist a non-compact semisimple subgroup of $G$ which acts properly on $G/H$. 
\end{itemize}
\end{theorem}

In other words, an analogue of Kobayashi's conjecture is \emph{false} in this setting: 
the above $G/H$ admits a proper action of a cocompact lattice of $\OO(n,1) \ (n =2,3,4)$, 
whereas it does not admit a \emph{standard} one. 

\subsection{Comparison to previous results}

When $\Gamma \cong \Z$ and $L \cong \R$, the problem was completely settled by Kobayashi in his Ph.D.\ thesis. 
In this case, the existence of proper actions is equivalent to the existence of standard ones:

\begin{fact}[{Kobayashi~\cite[\S 4]{Kob89}}]\label{fact:Calabi-Markus}
For a homogeneous space $G/H$ of reductive type, the following three conditions are equivalent: 
\begin{itemize}
\item \textup{(P-inf)} --- 
$G/H$ admits a proper action of an infinite discrete subgroup of $G$. 
\item \textup{(P-Z)} --- 
$G/H$ admits a proper action of a discrete subgroup of $G$ isomorphic to $\Z$. 
\item \textup{(P-R)} --- 
$G/H$ admits a proper action of a reductive subgroup of $G$ isomorphic to $\R$. 
\end{itemize}
Furthermore, these conditions are also equivalent to $\rank_\R G > \rank_\R H$,
where $\rank_\R$ stands for the real rank of reductive Lie groups. 
\end{fact}

Fact~\ref{fact:Calabi-Markus} is called the \emph{Calabi--Markus phenomenon} 
because this kind of phenomenon was first found in their work on Lorentzian geometry~\cite{CM62}. 

\begin{remark}
The above terminologies (P-inf), (P-Z) and (P-R) are not conventional at all. 
We just introduced them in order to facilitate the comparison of various results in this area. 
The same applies to the other conditions which we shall introduce below. 
\end{remark}

The `second smallest case', namely, the case where $L$ is locally isomorphic to $\SL(2,\R)$ 
(or equivalently, to $\OO(2,1)$), has been studied actively in the last three decades. 
In this case, the natural classes of $\Gamma$ to be considered are surface groups and free groups, 
as they are respectively the classes of cocompact and non-cocompact lattices of $\SL(2,\R)$, up to commensurability. 
Thus, let us introduce the following three conditions on a homogeneous space $G/H$ of reductive type: 
\begin{itemize}
\item (P-sl$_2$R) \ --- \ 
$G/H$ admits a proper action of a closed subgroup of $G$ locally isomorphic to $\SL(2,\R)$.
\item (P-surf) \ --- \ 
$G/H$ admits a proper action of a discrete subgroup of $G$ isomorphic to a surface group of genus $\geqslant 2$. 
\item (P-free) \ --- \ 
$G/H$ admits a proper action of a discrete subgroup of $G$ isomorphic to a non-abelian free group. 
\end{itemize}
We also consider: 
\begin{itemize}
\item (P-nva) \ --- \ 
$G/H$ admits a proper action of a non-virtually abelian discrete subgroup of $G$.
\end{itemize}
An easy observation is that the following implications hold: 
\[
(\text{P-sl}_2\text{R}) \ \Rightarrow \ 
(\text{P-surf}) \ \Rightarrow \ 
(\text{P-free}) \ \Rightarrow \
(\text{P-nva}) \ \Rightarrow \ 
(\text{P-inf}). 
\]
Also, it is easy to see that the condition (P-sl$_2$R) is equivalent to the following: 
\begin{itemize}
\item (P-ss) \ --- \ 
$G/H$ admits a proper action of a non-compact semisimple subgroup of $G$. 
\end{itemize}

Let us now review two important results on this problem, due to Benoist and Okuda, respectively. 
Firstly, Benoist proved the following result in the mid-90s:

\begin{fact}[Benoist~\cite{Ben96}]\label{fact:Ben96}
Let $G/H$ be a homogeneous space of reductive type. 
\begin{enumerate}[label = {\upshape (\arabic*)}]
\item The conditions \textup{(P-free)} and \textup{(P-nva)} are equivalent. Moreover, they are also equivalent to 
$\ab^{-w_0} \not\subset W \ab_H$.
\item The condition \textup{(P-inf)} does not imply \textup{(P-free)}. 
For example, $G/H = \SL(3,\R) / \SO(2,1)$ is \textup{(P-inf)} but not \textup{(P-free)}. 
\end{enumerate}
\end{fact}

We do not explain the meaning of the last condition in Fact~\ref{fact:Ben96}~(1), as it is not used in this paper. 
The important point is that it is completely root-theoretic and easily checkable for a given $G/H$. 
Unlike Kobayashi's result mentioned above, proper actions constructed by Benoist are usually non-standard: 
they come from a version of the ping-pong lemma. 

Much later, Okuda proved the following result: 

\begin{fact}[Okuda~\cite{Oku16}]\label{fact:Okuda}
The condition \textup{(P-free)} does not imply \textup{(P-sl$_2$R)}. More precisely, if $G = \SL(5,\R)$ and 
\[
H = \exp (\h), \qquad \h = \left\{ \diag(t_1, \dots, t_5) \ \middle| \ \sum_{i=1}^5 t_i = 0,\ 3t_1+2t_2+t_3 = 0 \right\},
\]
then $G/H$ is \textup{(P-free)} but not \textup{(P-sl$_2$R)}. 
\end{fact}

In Subsection~\ref{subsect:ques-surface}, 
we review other previously known examples of homogeneous spaces which are (P-free) but not (P-sl$_2$R). 

\begin{remark}\label{rmk:Okuda-symmetric}
Okuda proved earlier in \cite{Oku13} that, if $H$ is a \emph{symmetric subgroup} of $G$ 
(i.e.\ is locally a fixed-point subgroup of some involution on $G$), 
then (P-sl$_2$R) is equivalent to $\ab^{-w_0} \not\subset W \ab_H$. 
Thus, the five conditions (P-ss), (P-sl$_2$R), (P-surf), (P-free), and (P-nva) are all equivalent in this case. 
\end{remark}

In summary, the following implications and non-implications were previously known: 
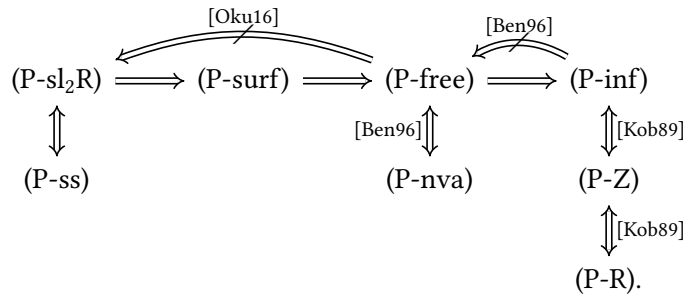
\[
\begin{tikzcd}
\textup{(P-sl$_2$R)} \ar[rr, bend left = 20, Leftarrow, "\diagup" marking, "\textup{\cite{Oku16}}"] \ar[d, Leftrightarrow] \ar[r, Rightarrow] & \textup{(P-surf)} \ar[r, Rightarrow] & \textup{(P-free)} \ar[d, Leftrightarrow, "\textup{\cite{Ben96}}"'] \ar[r, Rightarrow] \ar[r, bend left, Leftarrow, "\diagup" marking, "\textup{\cite{Ben96}}"] & \textup{(P-inf)} \ar[d, Leftrightarrow, "\textup{\cite{Kob89}}"] \\
\textup{(P-ss)} & & \textup{(P-nva)} & \textup{(P-Z)} \ar[d, Leftrightarrow, "\textup{\cite{Kob89}}"] \\
& & & \textup{(P-R)}.
\end{tikzcd}
\]

Our Theorem~\ref{thm:main-simplified} for $n=2$ now shows that (P-surf) does not imply (P-sl$_2$R). 
We still do not know if (P-free) is equivalent to (P-surf) (see Conjecture~\ref{conj:free-vs-surf}). 

\subsection{A detailed version of the main result}

Let us now explain our main results in more detail. 
For brevity, we introduce yet another notation: 

\begin{itemize}
\item (P-cocH$^n$) \ --- \ 
$G/H$ admits a proper action of a discrete subgroup of $G$ isomorphic to a cocompact lattice of $\OO(n,1)$. 
\end{itemize}

\begin{remark}
As we mentioned already, since $\OO(2,1)$ and $\SL(2,\R)$ are locally isomorphic, 
the condition (P-cocH$^2$) is equivalent to (P-surf).
\end{remark}

For $n = 2,3$, and $4$, we prove the following sufficient condition for (P-cocH$^n$): 

\begin{theorem}[{see \S~\ref{subsect:reductive} for notation}]\label{thm:main-general}
Let $G$ be a linear reductive Lie group and $H$ a reductive subgroup of $G$. 
Suppose that $L$ is a reductive subgroup of $G$ satisfying the following two properties: 
\begin{itemize}
\item $L$ is locally isomorphic to $\OO(n,1) \times \OO(N, 1)$, where $(n, N) = (2,3), (3,6)$, or $(4,8)$. 
\item $\ab_L \not\subset W \ab_H$. 
\end{itemize}
Then, there exists a discrete subgroup $\Gamma$ of $L$ 
which is isomorphic to a cocompact lattice of $\OO(n,1)$ and acts properly on $G/H$. 
In particular, $G/H$ is \textup{(P-cocH$^n$)} in this situation. 
\end{theorem}

In fact, there are continuously many such $\Gamma$ up to conjugation (see Remark~\ref{rmk:continuously-many}). 

\begin{remark}
If $\ab_L \cap W \ab_H = 0$, the above theorem is an immediate consequence of 
the properness criterion by Kobayashi and Benoist (Fact~\ref{fact:PROP}). 
The theorem says that a weaker condition $\ab_L \not\subset W \ab_H$ suffices. 
\end{remark}

Theorem~\ref{thm:main-general} enables us to construct an infinite series of homogeneous spaces of reductive type 
which are not (P-ss) but (P-cocH$^n$) for $n = 2,3$, and $4$: 

\begin{theorem}\label{thm:main-example}
For $m \geqslant 2$, let $G = \OO(m+1, m)$ or $\OO(2m+1, \C)$, and let $H$ be a reductive subgroup of $G$ such that
\[
\ab_H = \{ (t_1, \dots, t_m) \in \R^m \mid 2t_1 + t_2 + t_3 + \dots + t_{m-1} = 0 \} \subset \ab,
\]
where $\ab$ is identified with $\R^m$ in the standard way. 

\begin{enumerate}[label = {\upshape (\arabic*)}]
\item Let $(n, N) = (2, 3), (3, 6)$, or $(4, 8)$. 
If $m \geqslant N$, then $G/H$ is \textup{(P-cocH$^n$)}, 
i.e.\ there exists a discrete subgroup of $G$ which is isomorphic to a cocompact lattice of $\OO(n,1)$ and acts properly on $G/H$. 
\item If $m \geqslant 4$, then $G/H$ is not \textup{(P-ss)}, 
i.e.\ there does not exist a non-compact semisimple subgroup of $G$ which acts properly on $G/H$. 
\end{enumerate}
\end{theorem}

\begin{remark}
In the above theorem, $H$ can have a rather large semisimple part. 
For instance, one can take $H$ to be isomorphic to $\SL(m-2, \R) \times \OO(1,2) \times \R$ in the real case 
(and similarly, $\SL(m-2, \C) \times \OO(3, \C) \times \C^\times$ in the complex case). 
\end{remark}

Let us make a few comments on the proofs of Theorems~\ref{thm:main-general} and \ref{thm:main-example}. 

The proof of Theorem~\ref{thm:main-general} uses the theory of \emph{convex cocompact representations}. 
More precisely, utilizing Gu\'eritaud--Kassel's work~\cite{GK17}, 
we reduce our problem to finding a cocompact lattice $\Gamma$ of $\OO(n,1)$ and 
a small deformation of its standard convex cocompact representation 
\[
\Gamma \hookrightarrow \OO(n,1) \hookrightarrow \OO(N,1)
\]
which has a nice property (namely, being \emph{strictly dominated} by the original representation). 
Such a cocompact lattice and its deformation are known to exist for $(n,N) = (2,3), (3,6)$, and $(4,8)$ 
by Danciger--Gu\'eritaud--Kassel's work on right-angled Coxeter groups~\cite{DGK20}, 
hence we obtain Theorem~\ref{thm:main-general} for these values of $(n,N)$. 
Theorem~\ref{thm:main-example}~(1) is an easy consequence of Theorem~\ref{thm:main-general}.

\begin{remark}
Theorem~\ref{thm:main-general} can be seen as a generalization of Danciger--Gu\'eritaud--Kassel~\cite[Prop.~1.8]{DGK20} 
from $G/H = (\OO(N,1) \times \OO(N,1)) / {\Delta \OO(N,1)}$ to a more general homogeneous space. 
Note, however, that our setting is more complicated: 
one cannot directly obtain proper actions on $G/H$ from metric contractions on the hyperbolic space. 
\end{remark}

On the other hand, the proof of Theorem~\ref{thm:main-example}~(2) is based on 
the theory of \emph{nilpotent orbits} in semisimple Lie algebras. 
It reduces our tasks to show some combinatorial assertion (namely, (\hyperref[item:ast_m]{$\ast_m$}) in Section 4). 
That assertion is proved by a careful case-by-case analysis, which nevertheless is completely elementary.  

\subsection{Outline of the paper}

In Section~\ref{sect:prelim}, we summarize the previous results 
which are needed for the proofs of Theorems~\ref{thm:main-general} and \ref{thm:main-example}. 
We recall the structure theory of linear reductive Lie groups and 
the properness criterion due to Kobayashi and Benoist in Subsection~\ref{subsect:reductive}.
In Subsection~\ref{subsect:cc}, 
we first recall basics on convex cocompact representations and 
then review some results by Gu\'eritaud--Kassel~\cite{GK17} and Danciger--Gu\'eritaud--Kassel~\cite{DGK20}. 
The theory of nilpotent orbits and weighted Dynkin diagrams are summarized in Subsection~\ref{subsect:sl2}. 
We prove Theorems~\ref{thm:main-general} and \ref{thm:main-example}~(1) in Section~\ref{sect:proof-cocHn}, 
and then prove Theorem~\ref{thm:main-example}~(2) in Section~\ref{sect:proof-sl2}. 
These two sections can be read independently. 
Furthermore, Section~\ref{sect:proof-cocHn} does not rely on Subsection~\ref{subsect:sl2}, 
while Section~\ref{sect:proof-sl2} does not rely on Subsections~\ref{subsect:cc}. 
Finally, we mention some open questions in Section~\ref{sect:open-question}.

\section{Preliminaries}\label{sect:prelim}

\subsection{Linear reductive Lie groups}\label{subsect:reductive}

Let $G$ be a linear reductive Lie group and denote by $\g$ its Lie algebra. Choose a Cartan involution $\theta$ of $G$. 
We obtain a Cartan decomposition 
\[
\g = \kk \oplus \p, \qquad \kk = \{ X \in \g \mid \theta(X) = X \}, \qquad \p = \{ X \in \g \mid -\theta(X) = X \}, 
\]
and $\kk$ is the Lie algebra of the maximal compact subgroup 
$K = \{ k \in G \mid \theta(k) = k \}$ of $G$. 
Choose a maximal abelian subspace $\ab$ of $\p$ 
and denote by $\Sigma = \Sigma(\g, \ab)$ the restricted root system of $\g$. 
Choose a subset $\Delta \subset \Sigma$ of simple restricted roots and denote by 
$\Sigma^+ \subset \Sigma$ and $\overline{\ab^+} \subset \ab$ 
the set of positive restricted roots and the closed positive Weyl chamber with respect to $\Delta$, respectively. 

We allow $G$ to have multiple connected components, 
but we still assume that the root-theoretically defined restricted Weyl group, denoted as $W$, 
coincides with the group-theoretically defined one:
\[
W = N_K(\ab) / Z_K(\ab). 
\]
This always holds when $G$ is of inner type (and in particular, when $G$ is connected). 
Another example of such $G$ is $\OO(p,q)$ with $p \neq q$. 

Fix a $G$-invariant and $\theta$-invariant symmetric bilinear form such that $(X, Y) \mapsto -B(\theta X, Y)$
is positive definite (if $G$ is semisimple, one may take $B$ to be the Killing form). 
Let $\| {-} \| \colon \ab \to \R_{\geqslant 0}$ be the norm associated with $B|_{\ab \times \ab}$. 
We regard $\ab$ and its subset $\overline{\ab^+}$ as metric spaces by the distance function induced from this norm. 

A closed subgroup $H$ of $G$ is called a \emph{reductive subgroup} of $G$ 
if it is reductive as an abstract Lie group, and furthermore, $\theta(H) = H$ for some Cartan involution $\theta$ on $G$. 
The homogeneous space $G/H$ is said to be \emph{of reductive type} in this setting. 
It is known that any connected subgroup of $G$ which is semisimple as an abstract Lie group is 
closed and a reductive subgroup in $G$. 
If $G$ is equipped with a specified Cartan involution $\theta$ and a specified maximal abelian subspace $\ab$ of $\p$, 
every reductive subgroup $H$ of $G$ is implicitly assumed to satisfy the following two conditions: 
\begin{itemize}
\item $\theta(H) = H$. 
\item $\ab_H = \ab \cap \h$ is a maximal abelian subspace of $\p \cap \h$. 
\end{itemize}
Note that one can always achieve these conditions by replacing $H$ with its conjugate in $G$. 

We say that an element of $\g$ is \emph{nilpotent} (resp.\ \emph{hyperbolic}) if it is conjugate to some element in 
$\bigoplus_{\alpha \in \Sigma^+} \g_\alpha$
(resp.\ $\ab$). If $G$ is semisimple, then $X \in \g$ is nilpotent (resp.\ hyperbolic) precisely when 
$\ad(X) \colon \g \to \g$ is nilpotent (resp.\ diagonalizable over $\R$) as a linear endomorphism. 
A conjugacy class of nilpotent (resp.\ hyperbolic) elements in $\g$ is called 
a \emph{nilpotent} (resp.\ \emph{hyperbolic}) \emph{orbit}. 
The name comes from an obvious fact that conjugacy classes of elements in $\g$ are nothing but $\Ad(G)$-orbits. 

Let $\mu \colon G \to \overline{\ab^+}$ be the \emph{Cartan projection} on $G$, 
namely, $\mu(g)$ is the unique element of $\overline{\ab^+}$ whose exponential $\exp(\mu(g)) \in G$ is contained in $KgK$. 
We have the following criterion for the properness of actions: 

\begin{fact}[{Kobayashi~\cite{Kob89}, \cite[Th.~1.1]{Kob96}, Benoist~\cite[Th.~5.2]{Ben96}}]\label{fact:PROP}
Let $H$ and $L$ be two closed subgroups of $G$. Then, $L$ acts properly on $G/H$ if and only if 
the set 
\[
\mu(L) \cap \overline{N}(\mu(H), R) 
\]
is bounded in $\overline{\ab^+}$ for any $R \geqslant 0$, 
where $\overline{N}(-, R)$ signifies the closed $R$-neighbourhood. 
\end{fact}

For readability purpose, we shall write $\mu_j(\gamma)$ instead of $\mu(j(\gamma))$ whenever $j \colon \Gamma \to G$ is a group homomorphism and $\gamma \in \Gamma$. 

\subsection{Convex cocompact representations into \texorpdfstring{$\OO(n,1)$}{O(n, 1)}}\label{subsect:cc}

Throughout this subsection, let $\Gamma$ be a finitely generated group, and let 
$\lvert - \rvert \colon \Gamma \to \N$ 
be the word length function associated with a fixed finite and symmetric generating set of $\Gamma$. 
We write 
\[
\Hyp^n = \OO(n,1) / (\OO(n) \times \OO(1))
\]
for the $n$-dimensional real hyperbolic space. We always assume that $n \geqslant 2$ and that $\Gamma$ is infinite. 

\begin{definition}
A representation $j \colon \Gamma \to \OO(n,1)$ is called \emph{convex cocompact} 
if it has the following two properties: 
\begin{enumerate}[label = {\upshape (\roman*)}]
\item The kernel of $j$ is finite. 
\item There exists a nonempty $j(\Gamma)$-invariant closed convex subset of $\Hyp^n$ 
on which $j(\Gamma)$ acts cocompactly. 
\end{enumerate}
\end{definition}

We write $\Hom^\textup{cc}(\Gamma, \OO(n,1))$ for the subspace of $\Hom(\Gamma, \OO(n,1))$ 
consisting of convex cocompact representations. 

Let us recall two basic properties of convex cocompact representations: 

\begin{fact}[{see e.g.\ Bourdon~\cite[\S 1.8]{Bou95}}]\label{fact:cc-alternative-definition}
A representation $j \colon \Gamma \to \OO(n,1)$ is convex cocompact if and only if 
there exist $\ep, M > 0$ such that,
for any $\gamma \in \Gamma$, we have 
\[
\| \mu_j(\gamma) \| \geqslant \ep \lvert \gamma \rvert - M.
\]
\end{fact}

\begin{fact}[{see e.g.\ Bowditch~\cite[Prop.~4.1]{Bow98}}]\label{fact:cc-open}
The subspace $\Hom^\textup{cc}(\Gamma, \OO(n,1))$ is open in $\Hom(\Gamma, \OO(n,1))$.
\end{fact}

\begin{example}
Let $N \geqslant n \geqslant 2$, and let $\Gamma$ be a cocompact lattice of $\OO(n,1)$. 
It follows immediately from Fact~\ref{fact:cc-open} that any small deformation of the composite 
\[
\Gamma \hookrightarrow \OO(n,1) \hookrightarrow \OO(N,1),
\]
where the latter embedding is the standard block-diagonal one, 
is a convex cocompact representation of $\Gamma$ into $\OO(N,1)$. 
\end{example}

Now, let us briefly summarize (some part of) the work by Gu\'eritaud--Kassel~\cite{GK17}, 
which we shall use in Section~\ref{sect:proof-cocHn}. 
Note that, for $n=2$, the results below are proved earlier in Kassel's Ph.D.\ thesis~\cite[Chap.~5]{Kas09}: 

\begin{fact}[{Gu\'eritaud--Kassel~\cite[Lem.~4.7~(1)]{GK17}}]\label{fact:GK1}
Let $j \colon \Gamma \to \OO(n,1)$ be a convex cocompact representation and 
$\rho \colon \Gamma \to \OO(n,1)$ a representation. 
Then, there exists a $(j, \rho)$-equivariant Lipschitz map from $\Hyp^n$ to itself. 
\end{fact}

For $j$ and $\rho$ as above, define $C_\textup{Lip}(j, \rho) \geqslant 0$ to be the \emph{best Lipschitz constant}, i.e.\ 
the infimum of $c \geqslant 0$
for which there exists a $(j, \rho)$-equivariant $c$-Lipschitz map 
from $\Hyp^n$ to itself. 
We say that $\rho$ is \emph{strictly dominated} by $j$ if $C_\textup{Lip}(j, \rho) < 1$. 

\begin{remark}
In general, the infimum may not be attained. See \cite[\S 10.3]{GK17}. 
\end{remark}

\begin{fact}[{Gu\'eritaud--Kassel~\cite[Prop.~1.5]{GK17}}]\label{fact:GK2}
The following map is continuous: 
\[
C_\textup{Lip} \colon \Hom^\textup{cc}(\Gamma, \OO(n,1)) \times \Hom(\Gamma, \OO(n,1)) \to \R_{\geqslant 0}, \qquad 
(j, \rho) \mapsto C_\textup{Lip}(j, \rho).
\]
\end{fact}

In the above situation, we define 
$C_\textup{Car}(j, \rho) \in \R_{\geqslant 0}$ to be the infimum of $c \geqslant 0$ satisfying
\[
\sup_{\gamma \in \Gamma} \left( \| \mu_\rho(\gamma) \| - c \| \mu_j(\gamma) \| \right) < \infty. 
\]
It follows easily from the definition that $C_\textup{Lip}(j, \rho) \geqslant C_\textup{Car}(j, \rho)$. We thus have:

\begin{lemma}\label{lem:GK4}
Let $j \colon \Gamma \to \OO(n,1)$ be a convex cocompact representation and 
$\rho \colon \Gamma \to \OO(n,1)$ a representation. 
If $c > C_\textup{Lip}(j, \rho)$, there exists $M > 0$ such that, for any $\gamma \in \Gamma$, we have
\[
c \| \mu_j(\gamma) \| + M \geqslant \| \mu_\rho(\gamma) \|,
\]
where $\mu_j = \mu \circ j$ and $\mu_\rho = \mu \circ \rho$. 
\end{lemma}

\begin{proof}
This follows from the inequality $C_\textup{Lip}(j, \rho) \geqslant C_\textup{Car}(j, \rho)$ 
and the definition of $C_\textup{Car}$. 
\end{proof}

Furthermore, the following holds: 

\begin{fact}[{Gu\'eritaud--Kassel~\cite[\S 7 and Cor.~1.12]{GK17}}]\label{fact:GK3}
Let $j \colon \Gamma \to \OO(n,1)$ be a convex cocompact representation and 
$\rho \colon \Gamma \to \OO(n,1)$ a representation. 
Then, we have the following: 
\begin{enumerate}[label = {\upshape (\arabic*)}]
\item $C_\textup{Car}(j, \rho) < 1$ if and only if $\rho$ is strictly dominated by $j$ 
(i.e.\ $C_\textup{Lip}(j, \rho) < 1)$. 
\item If $C_\textup{Lip}(j, \rho) \geqslant 1$, then $C_\textup{Lip}(j, \rho) = C_\textup{Car}(j, \rho)$. 
\end{enumerate}
\end{fact}

From this, one immediately obtains the following: 

\begin{lemma}\label{lem:GK5}
For any convex cocompact representation $j \colon \Gamma \to \OO(n,1)$, 
we have $C_\textup{Lip}(j, j) = 1$. 
\end{lemma}

\begin{proof}
This follows from $C_\textup{Car}(j,j) = 1$ and Fact~\ref{fact:GK3}~(2). 
\end{proof}

Later, Danciger--Gu\'eritaud--Kassel constructed small deformations of convex cocompact right-angled reflection groups 
in $\OO(n,1)$ which are strictly dominated by the original ones. 
For cocompact lattices, their results are summarized as follows: 

\begin{fact}[{Danciger--Gu\'eritaud--Kassel~\cite[\S 4]{DGK20}}]\label{fact:DGK1}
Let $(n, N, P)$ be one of the following: 
\begin{itemize}
\item $n = 2$, $N = 3$, and $P$ is a convex right-angled $2k$-gon in $\Hyp^2$ for some $k \geqslant 3$. 
\item $n = 4$, $N = 8$, and $P$ is the regular right-angled $120$-cell in $\Hyp^4$. 
\end{itemize}
Let $\Gamma$ be the cocompact lattice of $\OO(n,1)$ generated by the reflections 
with respect to the facets of $P$ in $\Hyp^n$. 
Then, there exist $\ep > 0$ and a continuous family of convex cocompact representations 
\[
\rho_t \colon \Gamma \to \OO(N,1) \quad (-\ep < t < \ep)
\]
satisfying the following two conditions: 
\begin{itemize}
\item $\rho_0$ is the composite 
\[
\Gamma \hookrightarrow \OO(n,1) \hookrightarrow \OO(N,1),
\]
where the latter embedding is the standard block-diagonal one. 
\item For any $0 < s \leqslant t < \ep$, 
there exists a $(\rho_s, \rho_t)$-equivariant $\frac{\cosh s}{\cosh t}$-Lipschitz map from $\Hyp^N$ to itself.
\end{itemize}
\end{fact}

\begin{lemma}\label{lem:DGK2}
The conclusion of Fact~\ref{fact:DGK1} holds also for the following $(n, N, P)$: 
\begin{itemize}
\item $n = 3$, $N = 6$, and $P$ is a convex right-angled bounded polyhedron in $\Hyp^3$. 
\end{itemize}
\end{lemma}

\begin{proof}
This follows directly from the Appel--Haken four-colour theorem~\cite{AH77}, \cite{AHK77} 
and Danciger--Gu\'eritaud--Kassel~\cite[Prop.~4.1]{DGK20}. 
Of course, for a concrete choice of $P$ (such as the regular right-angled dodecahedron), 
one does not need to rely on the four-colour theorem. 
\end{proof}

\begin{lemma}\label{lem:DGK3}
Let $(n, N, P)$ be as in Fact~\ref{fact:DGK1} or Lemma~\ref{lem:DGK2}. 
Then, for any $0 < t < \ep$, we have 
\[
C_\textup{Lip}(\rho_0, \rho_t) \leqslant \frac{1}{\cosh t},
\]
and in particular, the representation $\rho_t$ is strictly dominated by $\rho_0$. 
\end{lemma}

\begin{proof}
By the continuity of $C_\textup{Lip}$ (Fact~\ref{fact:GK2}), we have 
\[
C_\textup{Lip}(\rho_0, \rho_t) 
= \lim_{s \searrow 0} C_\textup{Lip}(\rho_s, \rho_t) \leqslant \lim_{s \searrow 0} \frac{\cosh s}{\cosh t} 
= \frac{1}{\cosh t} < 1 \qquad (0 < t < \ep). \qedhere
\]
\end{proof}

\subsection{Weighted Dynkin diagrams and proper \texorpdfstring{$\SL(2,\R)$}{SL(2, R)}-actions}\label{subsect:sl2}

In this subsection, $G$ is assumed to be a linear \emph{semisimple} Lie group for simplicity. 
The standard textbook for the results in this subsection is Collingwood--McGovern~\cite{CM93}. 

We first observe that, since $G$ is linear, 
any Lie algebra homomorphism from $\ssl(2,\R)$ to $\g$ lifts uniquely to a Lie group homomorphism from $\SL(2,\R)$ to $G$. 
To put differently, the classification of connected subgroups in $G$ locally isomorphic to $\SL(2,\R)$ is equivalent to 
the classification of Lie subalgebras of $\g$ isomorphic to $\ssl(2,\R)$. 
It is customary to describe the latter in the following language:

\begin{definition}
A triple $(h,e,f)$ of elements in $\g$ is called an \emph{$\ssl_2$-triple} if 
\[
[h,e] = 2e, \qquad
[h,f] = -2f, \qquad
[e,f] = h.
\]
The elements $h,e,f \in \g$ are called 
the \emph{neutral}, \emph{nilpositive}, and \emph{nilnegative element} of the triple, respectively. 
\end{definition}

In other words, $(h,e,f)$ is called an $\ssl_2$-triple if the following is a Lie algebra homomorphism: 
\[
\varphi \colon \ssl(2,\R) \to \g, \qquad \begin{pmatrix} a & b \\ c & -a \end{pmatrix} \mapsto ah + be + cf. 
\]
Our convention is that the trivial triple $(h,e,f) = (0,0,0)$ is an $\ssl_2$-triple. 

We now recall two fundamental results on $\ssl_2$-triples: 

\begin{fact}[Jacobson--Morozov; see e.g.\ {\cite[Thms.~3.3.1 and 9.2.1]{CM93}}]\label{fact:Jacobson-Morozov}
Every nilpotent element in $\g$ is a nilpositive element of some $\ssl_2$-triple in $\g$. 
\end{fact}

\begin{fact}[Kostant; see e.g.\ {\cite[Thms.~3.4.10 and 9.2.3]{CM93}}]
Let $(h,e,f)$ and $(h',e,f')$ be two $\ssl_2$-triples in $\g$ with the same nilpositive element $e$. 
Then, there exists $g \in \Stab_G(e)$ such that $h' = \Ad(g) h$ and $f' = \Ad(g) f$, 
where $\Stab_G(e)$ is the stabilizer of $e$ in $G$. 
\end{fact}

These two facts imply that the assignment $(h,e,f) \mapsto e$ induces a bijection 
from the set of conjugacy classes of $\ssl_2$-triples in $\g$ to the set of nilpotent orbits in $\g$. 

Let $(h,e,f)$ be an $\ssl_2$-triple in $\g$. 
The hyperbolic orbit $\Ad(G)h$ depends only on the nilpotent orbit $\Ad(G)e$, 
hence we call it the \emph{distinguished hyperbolic orbit} associated with $\Ad(G)e$. 
It follows that each distinguished hyperbolic orbit in $\g$ intersects $\overline{\ab^+}$ at exactly one point, 
which we call the \emph{standard neutral element}. 
It is often presented by the following data: 

\begin{definition}
We define the \emph{weighted Dynkin diagram} of $h \in \overline{\ab^+}$ as follows: 
\begin{itemize}
\item As a mere diagram, it is the Dynkin diagram for the restricted root system $\Sigma$ of $\g$. 
\item The node corresponding to a simple restricted root $\alpha \in \Delta$ is labelled by $\alpha(h) \in \R_{\geqslant 0}$. 
\end{itemize}
\end{definition}

\begin{example}
Let $G = \SL(m,\R) \ (m \geqslant 2)$ and 
\[
h = \begin{pmatrix} h_1 \\ & \ddots \\ & & h_m \end{pmatrix}  \qquad\qquad 
\left( h_1 \geqslant \dots \geqslant h_m,\ \sum_{i=1}^m h_m = 0 \right).
\]
Then, the weighted Dynkin diagram of $h$ is as follows: 
\[
\begin{tikzcd}[row sep = small, column sep = large]
\bullet \ar[d, phantom, "h_1-h_2"] \ar[r, no head] & \bullet \ar[d, phantom, "h_2-h_3"] \ar[r, no head] & \cdots \ar[r, no head] & \bullet \ar[d, phantom, "{h_{m-1} - h_m.}"] \\
\ & \ & \ & \ 
\end{tikzcd}
\]
\end{example}

Let us now observe that, if $(h,e,f)$ is an $\ssl_2$-triple in $\g$ with $h \in \overline{\ab^+}$, 
then the Cartan projection of the corresponding reductive subgroup $L$ in $G$ is given by 
\[
\mu(L) = \R_{\geqslant 0} \cdot h.
\]
Thus, in the following three steps, one can verify if a given homogeneous space $G/H$ of reductive type is (P-sl$_2$R) 
(i.e.\ if $G$ has a closed subgroup which is locally isomorphic to $\SL(2,\R)$ and acts properly on $G/H$): 

\begin{enumerate}[label = {\upshape (\roman*)}]
\item Classify all nilpotent orbits in $\g$. 
\item For each nilpotent orbit in $\g$, compute the weighted Dynkin diagram of the corresponding standard neutral element. 
\item Check if every weighted Dynkin diagram obtained in (ii) represents an element in $W \ab_H$.
The homogeneous space $G/H$ is (P-sl$_2$R) precisely when the answer is `No'. 
\end{enumerate}

For a general semisimple Lie group $G$, these are highly complicated tasks. 
The solution to the task (i) is surveyed in \cite[Chap.~9]{CM93}. 
Okuda~\cite{Oku13} accomplished (ii) by using Satake diagrams systematically, 
and also (iii) when $H$ is a symmetric subgroup in $G$. 

In this paper, we only need to consider the case where $G = \OO(2m+1, \C)$ and $H$ is as in Theorem~\ref{thm:main-example}. 
In this case, the tasks (i) and (ii) are much easier than the general case.
Firstly, if $G$ is a complex classical Lie group, 
one can classify nilpotent orbits by relatively elementary linear algebra and combinatorics. 
In our case, the following holds: 

\begin{fact}[Gerstenhaber; see e.g.\ {\cite[Thm.~5.1.2]{CM93}}]\label{fact:typeB-1}
If $G = \OO(2m+1, \C) \ (m \geqslant 1)$, 
the set of nilpotent orbits in $\g = \oo(2m+1,\C)$ is canonically bijective to the set of partitions of $2m+1$ 
in which even numbers occur with even multiplicity. 
\end{fact}

Furthermore, the corresponding weighted Dynkin diagrams are also easy to describe: 

\begin{fact}[Springer--Steinberg; see e.g.\ {\cite[Lem.~5.3.3]{CM93}}]\label{fact:typeB-2}
Let $\mathbf{d} = \br{(2m+1)^{d_{2m+1}}, \dots, 2^{d_2}, 1^{d_1}}$ be a partition of $2m+1$ 
in which even numbers occur with even multiplicity, namely, 
\[
d_1, d_3, \dots, d_{2m+1} \in \N, \qquad d_2, d_4, \dots, d_{2m} \in 2\N, \qquad
\sum_{\ell=1}^{2m+1} d_\ell \ell = 2m+1. 
\tag{$\dagger_m$}\label{tag:partition}
\]
Form a sequence of integers 
\[
(h_1, \dots, h_m, 0, -h_m, \dots, -h_1) \qquad (h_1 \geqslant \dots \geqslant h_m \geqslant 0)
\]
by first concatenating the $2m+1$ sequences 
\[
(d_{i}-1,\, d_{i}-3,\,...,\, -d_i+3,\, -d_{i}+1) \qquad (1 \leqslant i \leqslant 2m+1)
\]
and then rearranging it in a decreasing order.
Then, the standard neutral element in $\g = \oo(2m+1,\C)$ obtained from $\mathbf{d}$ via Fact~\ref{fact:typeB-1} 
is presented by the following weighted Dynkin diagram:
\[
\begin{tikzcd}[row sep = small, column sep = large]
\bullet \ar[d, phantom, "h_1-h_2"] \ar[r, no head] & \bullet \ar[d, phantom, "h_2-h_3"] \ar[r, no head] & \cdots \ar[r, no head] & \bullet \ar[d, phantom, "h_{m-1}-h_m"] \ar[r, shift left = 0.4, no head] \ar[r, phantom, "{\rangle}"] \ar[r, shift right = 0.4, no head] & \bullet \ar[d, phantom, "{h_m.}"] \\
\ & \ & \ & \ & \
\end{tikzcd}
\]
\end{fact}

\section{Proofs of Theorems~\ref{thm:main-general} and \ref{thm:main-example}~(1)}\label{sect:proof-cocHn}

The main goal of this section is to prove Theorem~\ref{thm:main-general}. 
It is a direct consequence of a more general result (Proposition~\ref{prop:main-estimate-2}) and 
Danciger--Gu\'eritaud--Kassel's work recalled in Subsection~\ref{subsect:cc}. 
We also prove Theorem~\ref{thm:main-example}~(1) as an application of Theorem~\ref{thm:main-general} in this section. 

Throughout this section, let $N \geqslant n \geqslant 2$,
let $j \colon \Gamma \to \OO(n,1)$ be a convex cocompact representation, 
and denote by $J \colon \Gamma \to \OO(N,1)$ the composite
\[
\Gamma \xrightarrow{j} \OO(n, 1) \hookrightarrow \OO(N, 1),
\]
where the latter embedding is the standard block-diagonal one. 
We first prove the following result: 

\begin{proposition}\label{prop:main-estimate}
Let $\ell$ be a proper (i.e.\ $0$- or $1$-dimensional) linear subspace of $\ab_{\OO(n,1)} \oplus \ab_{\OO(N,1)}$. 
Then, for any representation $\rho \colon \Gamma \to \OO(N,1)$ sufficiently close to and strictly dominated by $J$, 
there exist $\ep, M > 0$ such that, for any $\gamma \in \Gamma$, 
\[
d(\mu_{j,\rho}(\gamma), \ell) \geqslant \ep \lvert \gamma \rvert - M,
\]
where $\mu_{j, \rho} = \mu \circ (j, \rho)$. 
\end{proposition}

\begin{proof}
One may assume that $N = n$ (and hence, $J = j$) and that $\ell$ is $1$-dimensional. 
By Fact~\ref{fact:cc-alternative-definition}, it is enough to prove the existence of $\ep', M' > 0$ such that, 
for any $\gamma \in \Gamma$, we have 
\[
d(\mu_{j,\rho}(\gamma), \ell) \geqslant \ep' \| \mu_j(\gamma) \| - M'. 
\]

Let us identify $\ab_{\OO(n,1)} \oplus \ab_{\OO(n,1)}$ (resp.\ $\overline{(\ab_{\OO(n,1)} \oplus \ab_{\OO(n,1)})^+}$) 
with $\R^2$ (resp.\ $\R_{\geqslant 0}^2$) in an obvious way, 
and denote by $s \in (-\infty, \infty]$ the slope of $\ell$. 

Firstly, if $-\infty < s < 0$, 
the line $\ell$ meets the first quadrant $\R_{\geqslant 0}^2$ 
only at the origin, as depicted in Figure~\ref{fig:-infty<s<0} below. We have
\[
d(\mu_{j,\rho}(\gamma), \ell) \geqslant d\big( (\mu_j(\gamma), 0), \ell \big),
\]
and the above inequality (with $M' = 0$) follows from this.

Next, suppose that $1 \leqslant s \leqslant \infty$. 
Since $\rho$ is strictly dominated by $j$, we can take $c \in \R$ so that $C_\textup{Lip}(j, \rho) < c < s$. 
By Lemma~\ref{lem:GK5}, there exists $M'' > 0$ such that, for any $\gamma \in \Gamma$, we have
\[
\| \mu_\rho(\gamma) \| \leqslant c \| \mu_j(\gamma) \| + M''. 
\]
The desired inequality follows easily from this, as depicted in Figure~\ref{fig:1<=s<=infty} below.

Let us finally consider the case where $0 \leqslant s < 1$.
By Facts~\ref{fact:cc-open}, \ref{fact:GK2} and Lemma~\ref{lem:GK5}, 
any representation $\rho \colon \Gamma \to \OO(n,1)$ sufficiently close to $j$ must be 
convex cocompact and satisfy $C_\textup{Lip}(\rho, j) < 1/s$. 
For such $\rho$, take $c' \in \R$ so that $C_\textup{Lip}(\rho, j) < c' < 1/s$. 
As before, one can derive the desired inequality from Lemma~\ref{lem:GK5}, 
as depicted in Figure~\ref{fig:0<=s<1} below.
\end{proof}

\begin{figure}[!hbt]
\minipage{0.33\textwidth}
\begin{tikzpicture}
\filldraw [fill=gray, opacity=0.3] (0,0) -- (0, 2.8) -- (2.8, 2.8) -- (2.8, 0) -- cycle;
\draw[->] (-0.3, 0) -- (2.9, 0);
\draw[->] (0, -0.3)-- (0, 2.9);
\draw[thick] (-0.6, 0.2) -- (0.9, -0.3);
\coordinate (L) at (-0.6, 0.3) node [left] at (L) {$\ell$};
\coordinate (X) at (2.9, 0) node [right] at (X) {$\| \mu_j(\gamma) \|$};
\coordinate (Y) at (0, 2.9) node [above] at (Y) {$\| \mu_\rho(\gamma) \|$};
\coordinate (1) at (0.3, 0.2); \fill (1) circle (1.5pt);
\coordinate (2) at (1.3, 2.2); \fill (2) circle (1.5pt);
\coordinate (3) at (2.2, 0.4); \fill (3) circle (1.5pt);
\coordinate (4) at (1.8, 1.1); \fill (4) circle (1.5pt);
\coordinate (5) at (0.7, 0.9); \fill (5) circle (1.5pt);
\coordinate (6) at (1.4, 1.9); \fill (6) circle (1.5pt);
\coordinate (7) at (1.4, 0.7); \fill (7) circle (1.5pt);
\coordinate (8) at (2.2, 2.1); \fill (8) circle (1.5pt);
\coordinate (9) at (0.4, 1.8); \fill (9) circle (1.5pt);
\coordinate (10) at (0.1, 2.4); \fill (10) circle (1.5pt);
\coordinate (11) at (2.5, 2.7); \fill (11) circle (1.5pt);
\coordinate (12) at (1.2, 1.5); \fill (12) circle (1.5pt);
\coordinate (13) at (0.8, 2.5); \fill (13) circle (1.5pt);
\coordinate (14) at (1.7, 0.3); \fill (14) circle (1.5pt);
\coordinate (15) at (2.6, 1.5); \fill (15) circle (1.5pt);
\coordinate (16) at (0.2, 1.2); \fill (16) circle (1.5pt);
\coordinate (17) at (0.9, 0.4); \fill (17) circle (1.5pt);
\coordinate (18) at (2.7, 0.7); \fill (18) circle (1.5pt);
\coordinate (19) at (1.9, 2.5); \fill (19) circle (1.5pt);\end{tikzpicture}
\caption{The $-\infty < s < 0$ case}\label{fig:-infty<s<0}
\endminipage\hfill
\minipage{0.32\textwidth}
\begin{tikzpicture}
\filldraw [fill=gray, opacity=0.3] (0, 0) -- (0, 0.75) -- (1.725, 2.8) -- (2.8, 2.8) -- (2.8, 0) -- cycle;
\draw[->] (-0.3, 0) -- (2.9, 0);
\draw[->] (0, -0.3)-- (0, 2.9);
\draw[thick] (-0.1, -0.3) -- (1, 3);
\draw (-0.2, 0.5) -- (1.8, 2.9);
\coordinate (L) at (1.1, 3) node [above] at (L) {$\ell$};
\coordinate (X) at (2.9, 0) node [right] at (X) {$\| \mu_j(\gamma) \|$};
\coordinate (Y) at (0, 2.9) node [above] at (Y) {$\| \mu_\rho(\gamma) \|$};
\coordinate (1) at (0.3, 0.2); \fill (1) circle (1.5pt);
\coordinate (2) at (1.5, 2.2); \fill (2) circle (1.5pt);
\coordinate (3) at (2.6, 0.6); \fill (3) circle (1.5pt);
\coordinate (4) at (1.8, 1.1); \fill (4) circle (1.5pt);
\coordinate (5) at (0.7, 0.9); \fill (5) circle (1.5pt);
\coordinate (6) at (1.4, 1.9); \fill (6) circle (1.5pt);
\coordinate (7) at (1.4, 0.7); \fill (7) circle (1.5pt);
\coordinate (8) at (2.2, 2.1); \fill (8) circle (1.5pt);
\coordinate (9) at (2.0, 2.6); \fill (9) circle (1.5pt);
\coordinate (10) at (0.1, 0.7); \fill (10) circle (1.5pt);
\coordinate (11) at (2.3, 0.4); \fill (11) circle (1.5pt);
\coordinate (12) at (1.2, 0.1); \fill (12) circle (1.5pt);
\coordinate (13) at (1.0, 1.5); \fill (13) circle (1.5pt);
\coordinate (14) at (2.6, 1.4); \fill (14) circle (1.5pt);
\coordinate (15) at (2.5, 2.7); \fill (15) circle (1.5pt);
\coordinate (16) at (0.3, 1.0); \fill (16) circle (1.5pt);\end{tikzpicture}
\caption{The $1 \leqslant s \leqslant \infty$ case}\label{fig:1<=s<=infty}
\endminipage\hfill
\minipage{0.32\textwidth}
\begin{tikzpicture}
\filldraw [fill=gray, opacity=0.3] (0, 0) -- (0.75, 0) -- (2.8, 1.725) -- (2.8, 2.8) -- (0, 2.8) -- cycle;
\draw[->] (-0.3, 0) -- (2.9, 0);
\draw[->] (0, -0.3)-- (0, 2.9);
\draw[thick] (-0.3, -0.1) -- (3, 1);
\draw (0.5, -0.2) -- (2.9, 1.8);
\coordinate (L) at (3, 1.1) node [right] at (L) {$\ell$};
\coordinate (X) at (2.9, 0) node [right] at (X) {$\| \mu_j(\gamma) \|$};
\coordinate (Y) at (0, 2.9) node [above] at (Y) {$\| \mu_\rho(\gamma) \|$};
\coordinate (1) at (0.2, 0.3); \fill (1) circle (1.5pt);
\coordinate (2) at (2.2, 1.5); \fill (2) circle (1.5pt);
\coordinate (3) at (0.6, 2.6); \fill (3) circle (1.5pt);
\coordinate (4) at (1.1, 1.8); \fill (4) circle (1.5pt);
\coordinate (5) at (0.9, 0.7); \fill (5) circle (1.5pt);
\coordinate (6) at (1.9, 1.4); \fill (6) circle (1.5pt);
\coordinate (7) at (0.7, 1.4); \fill (7) circle (1.5pt);
\coordinate (8) at (2.1, 2.2); \fill (8) circle (1.5pt);
\coordinate (9) at (2.6, 2.0); \fill (9) circle (1.5pt);
\coordinate (10) at (0.7, 0.1); \fill (10) circle (1.5pt);
\coordinate (11) at (0.4, 2.3); \fill (11) circle (1.5pt);
\coordinate (12) at (0.1, 1.2); \fill (12) circle (1.5pt);
\coordinate (13) at (1.5, 1.0); \fill (13) circle (1.5pt);
\coordinate (14) at (1.4, 2.6); \fill (14) circle (1.5pt);
\coordinate (15) at (2.7, 2.5); \fill (15) circle (1.5pt);
\coordinate (16) at (1.0, 0.3); \fill (16) circle (1.5pt);\end{tikzpicture}
\caption{The $0 \leqslant s < 1$ case}\label{fig:0<=s<1}
\endminipage
\end{figure}

\begin{remark}
In the above proof, the assumption $C_\textup{Lip}(j, \rho) < 1$ is crucial only when $s = 1$: 
if $1 < s \leqslant \infty$, one can argue as in the case of $0 \leqslant s < 1$. 
\end{remark}

Now, let $G$ be a linear real reductive Lie group, 
and let $L$ be a reductive subgroup of $G$ locally isomorphic to $\OO(n,1) \times \OO(N,1)$.
Since the identity component of $\OO(n,1) \times \OO(N,1)$ is a semisimple Lie group with trivial centre, 
replacing $L$ with its identity component if necessary, 
one may assume without loss of generality the existence of a local isomorphism
\[
f \colon L \to \OO(n,1) \times \OO(N,1). 
\]
For a representation $\rho \colon \Gamma \to \OO(N,1)$, define a discrete subgroup $\widetilde{\Gamma}_{j,\rho}$ of $L$ by
\[
\widetilde{\Gamma}_{j,\rho} = f^{-1}((j,\rho)(\Gamma)).
\]
We remark that any torsion-free finite-index subgroup of $\widetilde{\Gamma}_{j,\rho}$, 
which exists by Selberg's lemma~\cite[Lem.~8]{Sel60}, is also a finite-index subgroup of $\Gamma$. 

\begin{proposition}\label{prop:main-estimate-2}
Let $H$ be a reductive subgroup of $G$ satisfying $\ab_L \not\subset W \ab_H$.
If $\rho$ is sufficiently close to and strictly dominated by $J$, there exist 
$\ep, M > 0$ such that, for any $\widetilde{\gamma} \in \widetilde{\Gamma}_{j,\rho}$, we have 
\[
d(\mu(\widetilde{\gamma}), \mu(H)) \geqslant \ep \lvert \widetilde{\gamma} \rvert - M
\]
(and in particular, $\widetilde{\Gamma}_{j,\rho}$ acts properly on $G/H$ by Fact~\ref{fact:PROP}). 
\end{proposition}

\begin{proof}
We use the following elementary lemma: 

\begin{lemma}\label{lem:linear-space-estimate}
Let $V$ be a finite-dimensional real vector space, and let $V'$ and $V''$ be two linear subspaces of $V$. 
Fix an inner product on $V$, and denote by $d \colon V \times V \to \R_{\geqslant 0}$ the associated distance function. 
Then, there exists $\delta > 0$ such that, for any $v' \in V'$, we have 
\[
d(v', V'') \geqslant \delta \cdot d(v', V' \cap V''). 
\]
\end{lemma}

\begin{proof}[Proof of Lemma~\ref{lem:linear-space-estimate}]
Let $W$ be the orthogonal complement of $V' \cap V''$ in $V'$,
and write $\pi \colon V' \to W$ for the orthogonal projection.
Define $\delta > 0$ by
\[
\delta = \min_{w \in W,\, \| w \| = 1} d(w, V'').
\]
For any $v' \in V'$, we have 
\[
d(v', V'') = d(\pi(v'), V'') \geqslant \delta \| \pi(v') \|, \qquad
d(v', V' \cap V'') = d(\pi(v'), V' \cap V'') = \| \pi(v') \|.
\]
The lemma is thus proved.
\end{proof}

To prove Proposition~\ref{prop:main-estimate-2}, it is enough to show that there exist
$\ep', M' > 0$ such that, for any $\gamma \in \Gamma$ and $w \in W$, we have 
\[
d\left(\mu_{j,\rho}(\gamma), w\ab_H \right) \geqslant \ep' \lvert \gamma \rvert - M'.
\]
Here, we have regarded $\mu_{j,\rho}(\gamma) \in \ab_{\OO(n,1)} \oplus \ab_{\OO(N,1)}$ as an element of $\ab_L \subset \ab$ via the local isomorphism $f$. 
Applying Lemma~\ref{lem:linear-space-estimate} to $\ell = \ab_L \cap w\ab_H$ and then Proposition~\ref{prop:main-estimate} to 
$(V, V', V'') = (\ab, \ab_L, w\ab_H)$, we obtain this inequality.
\end{proof}

We can now prove Theorems~\ref{thm:main-general} and \ref{thm:main-example}~(1): 

\begin{proof}[Proof of Theorem~\ref{thm:main-general}]
Let $(n,N,P)$ be as in Fact~\ref{fact:DGK1} or Lemma~\ref{lem:DGK2}, 
and define $\Gamma$ and $(\rho_t)_{-\ep < t < \ep}$ accordingly. 
We write $j \colon \Gamma \to \OO(n,1)$ for the inclusion map. 
By Lemma~\ref{lem:DGK3} and Proposition~\ref{prop:main-estimate-2}, for a sufficiently small $t > 0$, 
the discrete subgroup $\widetilde{\Gamma}_{j,{\rho_t}}$ of $L$ acts properly on $G/H$. 
\end{proof}

\begin{proof}[Proof of Theorem~\ref{thm:main-example}~\textup{(1)}]
Since $\OO(2m+1, \C)$ and its split real form $\OO(m+1,m)$ have the same restricted root system, 
it suffices to prove the theorem for $G = \OO(m+1,m)$.
Consider the reductive subgroup $L = \OO(1,m-1) \times \OO(m,1)$ of $G$. 
Under the standard identification $\ab_{\OO(m+1,m)} = \R^m$, we have 
\[
\ab_L = \{ (t_1, t_2, 0, \dots, 0) \mid t_1, t_2 \in \R \}. 
\]
It follows that $\ab_L \not\subset W \ab_H$. Indeed, $(1, 3, 0, \dots, 0) \in \ab_L \smallsetminus W \ab_H$. 
We can thus apply Theorem~\ref{thm:main-general}. 
\end{proof}

\begin{remark}\label{rmk:sharp}
By construction, the above discrete subgroup $\widetilde{\Gamma}_{j,{\rho_t}}$ (with $t$ sufficiently small) is 
\emph{sharply embedded} in $G$ with respect to $H$ in the sense of Kassel--Tholozan~\cite[Defn.~1.6]{KT24+}. 
\end{remark}

\begin{remark}\label{rmk:continuously-many}
Take an arbitrarily small $\varepsilon > 0$. 
Let us see that, even up to conjugation, there are continuously many discrete subgroups of $G$ of the form 
$\widetilde{\Gamma}_{j, \rho_t} \ (0 < t < \ep)$. By the definition of $C_\textup{Lip}$, we have 
\[
C_\textup{Lip}(\rho_0, \rho_t) \cdot C_\textup{Lip}(\rho_t, \rho_0) 
\leqslant C_\textup{Lip}(\rho_0, \rho_0) \qquad (0 \leqslant t < \ep). 
\]
This together with Lemmas~\ref{lem:GK5} and \ref{lem:DGK3} yields 
\[
C_\textup{Lip}(\rho_t, \rho_0) \geqslant \cosh t \ (\geqslant 1) \qquad (0 \leqslant t < \ep). 
\]
Now, by Facts~\ref{fact:GK1} and \ref{fact:GK3}~(2), the following function is continuous: 
\[
\lbrack 0, \ep) \to \lbrack 1, \infty), \qquad t \mapsto C_\textup{Car}(\rho_t, \rho_0).
\]
By the intermediate value theorem, for each $1 < c < \cosh \ep$, there exists $0 < t(c) < \ep$ such that $C_\textup{Car}(\rho_{t(c)}, \rho_0) = c$. 
Now, if $1 < c < c' < \cosh \ep$,
the \emph{asymptotic cone} (Benoist~\cite[\S 4]{Ben97}) to $\mu_{j, \rho_{t(c)}}(\Gamma)$ does not contain 
$(c', 1) \in \R_{\geqslant 0}^2$, while the asymptotic cone to $\mu_{j, \rho_{t(c)}}(\Gamma)$ does. 
This implies that $\mu(\widetilde{\Gamma}_{j, \rho_{t(c)}})$ and $\mu(\widetilde{\Gamma}_{j, \rho_{t(c')}})$ 
have different asymptotic cones in $\overline{\ab^+}$. 
In particular, they are not conjugate in $G$. 
\end{remark}

\begin{remark}
Let us say that a subgroup of $\OO(n,1)$ is a \emph{convex cocompact right-angled reflection group} 
if it is generated by the reflections with respect to the facets of $P$ in $\Hyp^n$, 
where $P$ is a (possibly unbounded) right-angled polytope in $\Hyp^n$ with finitely many facets and without ideal vertices. 
Note that such a subgroup is always convex cocompact (see Desgroseilliers--Haglund~\cite[Thm.~4.7]{DH13}). 
From Proposition~\ref{prop:main-estimate-2} and Danciger--Gu\'eritaud--Kassel~\cite[Prop.~4.1]{DGK20},
one can deduce the following variant of Theorems~\ref{thm:main-general} and \ref{thm:main-example}~(1): 
\end{remark}

\begin{theorem}\label{thm:right-angled}
Let $\Gamma$ be a convex cocompact right-angled reflection group in $\OO(n,1)$ 
associated with a polytope with $f$ facets. 
Let $G$ be a linear reductive Lie group and $H$ a reductive subgroup of $G$. 
Suppose that $L$ is a reductive subgroup of $G$ with the following two properties: 
\begin{itemize}
\item $L$ is locally isomorphic to $\OO(n,1) \times \OO(n+f, 1)$. 
\item $\ab_L \not\subset W \ab_H$. 
\end{itemize}
Then, there exists a discrete subgroup of $L$ which is commensurable to $\Gamma$ and acts properly on $G/H$. 
In particular, the homogeneous space $G/H$ in Theorem~\ref{thm:main-example} 
admits a proper action of a discrete subgroup of $G$ commensurable to $\Gamma$ whenever $m \geqslant n+f$. 
\end{theorem}

\section{Proof of Theorem~\ref{thm:main-example}~(2)}\label{sect:proof-sl2}

In this section, we prove Theorem~\ref{thm:main-example}~(2). 

\begin{proof}[Proof of Theorem~\ref{thm:main-example}~(2)]
Since $\OO(2m+1, \C)$ and its split real form $\OO(m+1,m)$ have the same restricted root system, 
it suffices to prove the theorem for $G = \OO(2m+1,\C) \ (m \geqslant 4)$.
We put $\g = \oo(2m+1,\C)$ and use the standard identifications
\[
\ab = \R^m, \qquad W = \Sym_m \ltimes \{ \pm 1 \}^m, \qquad 
\overline{\ab^+} = \{ (t_1, \dots, t_m) \in \R^m \mid t_1 \geqslant \dots \geqslant t_m \geqslant 0 \}. 
\]

We first observe that the weighted Dynkin diagram given in Fact~\ref{fact:typeB-2} corresponds to 
\[
(\underbrace{2m, \dots, 2m}_{a_{2m}}, \dots, \underbrace{1, \dots, 1}_{a_1}, \underbrace{0, \dots, 0}_{a_0}) \in \overline{\ab^+}, 
\]
where 
\[
a_0 = \frac{1}{2}\left( \sum_{i=0}^m d_{2i+1} - 1 \right), \qquad a_{2k} = \sum_{i=k}^m d_{2i+1} \quad (1 \leqslant k \leqslant m), \qquad 
a_{2k-1} = \sum_{i=k}^m d_{2i} \quad (1 \leqslant k \leqslant m). 
\]
The condition (\ref{tag:partition}) in Fact~\ref{fact:typeB-2} is rewritten in terms of $a_i$'s as follows: 
\begin{gather*}
a_0, a_2, \dots, a_{2m} \in \N, \qquad 
a_1, a_3, \dots, a_{2m-1} \in 2\N, \\
2a_0+1 \geqslant a_2 \geqslant a_4 \geqslant \dots \geqslant a_{2m}, \qquad 
a_1 \geqslant a_3 \geqslant \dots \geqslant a_{2m-1}, \qquad \sum_{i=0}^{2m} a_i = m. 
\tag{$\ddag_m$}\label{tag:rewritten}
\end{gather*}

Thus, to prove Theorem~\ref{thm:main-example}~(2), 
it is enough to show the following assertion for every $m \geqslant 4$: 

\begin{itemize}[label = $(\ast_m)$]
\item \label{item:ast_m} \textit{Let $\mathbf{a} = (a_0, \dots, a_{2m})$ 
be a sequence satisfying \textup{(\ref{tag:rewritten})}, and put 
\[
v_{\mathbf{a}} = (\underbrace{2m, \dots, 2m}_{a_{2m}}, \dots, \underbrace{1, \dots, 1}_{a_1}, \underbrace{0, \dots, 0}_{a_0}) \in \R^m.
\]
Then, there exists $\sigma \in \Sym_m \ltimes \{ \pm 1 \}^m$ such that 
\[
\left\langle v_\mathbf{a},\, \sigma (2, \underbrace{1, \dots, 1}_{m-2}, 0) \right\rangle = 0,
\]
where $\langle {-}, {-} \rangle$ is the standard inner product on $\R^m$.}
\end{itemize}

Assume that (\hyperref[item:ast_m]{$\ast_m$}) were false for some $m \geqslant 4$, and let 
$\mathbf{a} = (a_0, \dots, a_{2m})$ be one of the counterexamples. 
Without loss of generality, 
we may assume that (\hyperref[item:ast_m]{$\ast_{m'}$}) is true for $4 \leqslant m' \leqslant m-1$. 
We shall prove in eight steps that such a sequence $\mathbf{a} = (a_0, \dots, a_{2m})$ cannot exist.

\textbf{Step 1}: Firstly, we must have $m \geqslant 5$, 
as the following case-by-case analysis shows that (\hyperref[item:ast_m]{$\ast_4$}) is true: 

\begin{center}
\begin{longtable}{c|c} \toprule
$v_\mathbf{a}$ & $\sigma (2,1,1,0)$ \\ \hline
$(0,0,0,0)$ & $(2,1,1,0)$ \\ \hline
$(2,0,0,0)$ & $(0,2,1,1)$ \\ \hline
$(1,1,0,0)$ & $(1,-1,2,0)$ \\ \hline
$(2,2,0,0)$ & $(1,-1,2,0)$ \\ \hline
$(4,2,0,0)$ & $(1,-2,1,0)$ \\ \hline
$(2,1,1,0)$ & $(1,-2,0,1)$ \\ \hline
$(2,2,2,0)$ & $(1,1,-2,0)$ \\ \hline
$(4,2,2,0)$ & $(1,-2,0,1)$ \\ \hline
$(6,4,2,0)$ & $(0,1,-2,1)$ \\ \hline
$(1,1,1,1)$ & $(1,1,-2,0)$ \\ \hline
$(4,2,1,1)$ & $(1,-1,-2,0)$ \\ \hline
$(3,3,1,1)$ & $(1,0,-1,-2)$ \\ \hline
$(8,6,4,2)$ & $(0,1,-2,1)$ \\ 
\bottomrule
\caption{The $m=4$ case}
\label{table:m=4}
\end{longtable}
\end{center}

\textbf{Step 2}: Let us show that our counterexample $\mathbf{a} = (a_0, \dots, a_{2m})$ 
must satisfy $a_2 = 2a_0$ or $2a_0+1$. 
Suppose otherwise, and define the sequence $\mathbf{a}' = (a'_0, \dots, a'_{2m-2})$ by 
\[
a'_i = \begin{cases} a_i-1 & (i = 0), \\ a_i & (\text{otherwise}), \end{cases}
\]
which satisfies (\hyperref[tag:rewritten]{$\ddag_{m-1}$}) 
(we remark that our assumption $a_2 \leqslant 2a_0 - 1$ implies $a_0 \geqslant 1$). 
By the induction hypothesis, one can apply (\hyperref[item:ast_m]{$\ast_{m-1}$}) to $\mathbf{a}'$ to see that 
\[
\left\langle v_{\mathbf{a}'},\, \sigma' (2, \underbrace{1, \dots, 1}_{m-3}, 0) \right\rangle = 0
\]
holds for some $\sigma' \in \Sym_{m-1} \ltimes \{ \pm 1 \}^{m-1}$.
Assigning $1$ to the remaining $0$ in $\mathbf{a}$, one obtains $\sigma \in \Sym_m \ltimes \{ \pm 1 \}^m$ satisfying 
\[
\left\langle v_\mathbf{a},\, \sigma (2, \underbrace{1, \dots, 1}_{m-2}, 0) \right\rangle = 0,
\]
which is a contradiction. 

\textbf{Step 3}: We now see that $m \geqslant 6$. 
The following table exhausts all possibilities for (\hyperref[item:ast_m]{$\ast_5$}) that are not excluded by Step 2: 

\begin{center}
\begin{longtable}{c|c} \toprule
$v_\mathbf{a}$ & $\sigma (2,1,1,1,0)$ \\ \hline
$(4,2,2,2,0)$ & $(1,-1,-1,0,2)$ \\ \hline
$(2,2,1,1,0)$ & $(1,0,-1,-1,2)$ \\ \hline
$(4,4,2,2,0)$ & $(1,0,-1,-1,2)$ \\ \hline
$(6,4,2,2,0)$ & $(0,1,-1,-1,2)$ \\ \hline
$(2,1,1,1,1)$ & $(1,1,-1,-2,0)$ \\ \hline
$(3,3,2,1,1)$ & $(1,-1,1,-2,0)$ \\ \hline
$(6,4,2,1,1)$ & $(0,1,-2,1,-1)$ \\ \hline
$(10,8,6,4,2)$ & $(1,-1,1,-2,0)$ \\
\bottomrule
\caption{The $m=5$ case}
\label{table:m=5}
\end{longtable}
\end{center}

\textbf{Step 4}: Let us show that our sequence $\mathbf{a} = (a_0, \dots, a_{2m})$ 
must satisfy $a_{2i-1} = 0$ for all $1 \leqslant i \leqslant m$. 
Suppose otherwise, and let $i_0$ be the largest number satisfying $a_{2i_0-1} \geqslant 2$. 
Then, the sequence $\mathbf{a}' = (a'_0, \dots, a'_{2m-4})$ defined by 
\[
a'_i = \begin{cases} a_i - 2 & (i = 2i_0-1), \\ a_i & (\text{otherwise}) \end{cases}
\]
satisfies (\hyperref[tag:rewritten]{$\ddag_{m-2}$}). 
By the induction hypothesis, one can apply (\hyperref[item:ast_m]{$\ast_{m-2}$}) to $\mathbf{a}'$ to see that 
\[
\left\langle v_{\mathbf{a}'},\, \sigma' (2, \underbrace{1, \dots, 1}_{m-4}, 0) \right\rangle = 0
\]
for some $\sigma' \in \Sym_{m-2} \ltimes \{ \pm 1 \}^{m-2}$. 
Now, assigning $1$ and $-1$ to the remaining two $(2i_0-1)$'s, 
one obtains $\sigma \in \Sym_m \ltimes \{ \pm 1 \}^m$ satisfying 
\[
\left\langle v_\mathbf{a},\, \sigma (2, \underbrace{1, \dots, 1}_{m-2}, 0) \right\rangle = 0,
\]
which is a contradiction. 

\textbf{Step 5}: Let $i_0$ be the largest number satisfying $a_{2i_0} \geqslant 1$. 
For the same reason as Step 4, we must have $a_{2i_0} = 1$. 
One can similarly prove 
\[
0 \leqslant a_{2i-2} - a_{2i} \leqslant 1 \qquad (2 \leqslant i \leqslant i_0).
\]

\textbf{Step 6}: We now see that $m \geqslant 8$ must hold. 
Indeed, the following tables exhaust all possibilities for 
(\hyperref[item:ast_m]{$\ast_6$}) and (\hyperref[item:ast_m]{$\ast_7$}) that are not excluded by Steps 2, 4, and 5: 

\begin{center}
\begin{longtable}{c|c} \toprule
$v_\mathbf{a}$ & $\sigma (2,1,1,1,1,0)$ \\ \hline
$(6,4,4,2,2,0)$ & $(0,1,-1,1,-1,2)$ \\ \hline
$(8,6,4,2,2,0)$ & $(1,-1,-1,1,0,2)$ \\ \hline
$(12,10,8,6,4,2)$ & $(0,1,-1,1,-1,-2)$ \\
\bottomrule
\caption{The $m=6$ case}
\label{table:m=6}
\end{longtable}
\begin{longtable}{c|c} \toprule
$v_\mathbf{a}$ & $\sigma (2,1,1,1,1,1,0)$ \\ \hline
$(6,4,4,2,2,2,0)$ & $(0,1,-1,1,-2,1,1)$ \\ \hline
$(8,6,4,4,2,2,0)$ & $(0,1,-1,1,-1,-2,1)$ \\ \hline
$(10,8,6,4,2,2,0)$ & $(0,1,-1,1,-1,-2,1)$ \\ \hline
$(14,12,10,8,6,4,2)$ & $(0,1,1,-1,-1,-1,-2)$ \\
\bottomrule
\caption{The $m=7$ case}
\label{table:m=7}
\end{longtable}
\end{center}

\textbf{Step 7}: Let us show that our sequence $\mathbf{a} = (a_0, \dots, a_{2m})$ 
must satisfy $a_{2i} = 0$ for $4 \leqslant i \leqslant m$. 
Suppose otherwise, and let $i_0$ be the largest number satisfying $a_{2i_0} \geqslant 1$. 
Then, the sequence
$\mathbf{a}' = (a'_0, \dots, a'_{2m-8})$ defined by 
\[
a'_i = \begin{cases} a_i - 1 & (i = 2i_0-6, 2i_0-4, 2i_0-2, 2i_0), \\ a_i & (\text{otherwise}) \end{cases}
\]
satisfies (\hyperref[tag:rewritten]{$\ddag_{m-4}$}). 
Arguing as before and observing 
\[
\langle (2i_0-6,2i_0-4,2i_0-2,2i_0),\, (1,-1,-1,1) \rangle = 0,
\]
one can see that (\hyperref[item:ast_m]{$\ast_m$}) is true for $\mathbf{a}'$, which is a contradiction. 

\textbf{Step 8}: By Steps 4 and 7, our sequence $\mathbf{a} = (a_0, \dots, a_{2m})$ 
satisfies $a_i = 0$ except possibly for $i=0,2,4,6$. 
By Steps 2 and 5, we have 
\[
8 \leqslant m = a_6 + a_4 + a_2 + a_0 \leqslant 1 + 2 + 3 + 1 = 7,
\]
which is absurd. The proof is now completed. 
\end{proof}

\section{Some open questions}\label{sect:open-question}

In this section, we discuss some open questions raised by this work. 

\subsection{On proper actions of surface groups}\label{subsect:ques-surface}

Concerning connections between the conditions 
(P-sl$_2$R), (P-surf), (P-free), (P-nva), and (P-inf), the last remaining problem is the following:

\begin{conjecture}\label{conj:free-vs-surf}
The condition \textup{(P-free)} does not imply \textup{(P-surf)}, 
namely, there exists a homogeneous space $G/H$ of reductive type 
which admits a proper action of a non-abelian free discrete subgroup of $G$ 
but does not admit a proper action of a discrete surface subgroup of $G$ of genus $\geqslant 2$. 
\end{conjecture}

As far as the authors know, there are no available tools to prove the failure of (P-surf) 
other than showing $\ab^{-w_0} \subset W\ab_H$ (i.e.\ showing the failure of (P-free)). 
We need a new necessary condition for a homogeneous space to be (P-surf) to solve this conjecture. 

Before our work, there were four known examples of homogeneous spaces of reductive type 
which are (P-free) but not (P-sl$_2$R): 
\begin{enumerate}[label = (\roman*)]
\item Okuda's first example which we restated in Fact~\ref{fact:Okuda}. 
\item An example given in the first author's paper~\cite[Ex.~1]{Boc17}, where 
$(\g, \h) = (\oo(4,4),\, \ssl(2,\R) \oplus \R)$.
\item Two examples given in the first author's joint work~\cite[Thm.~2]{BdGJT22+} 
with de Graaf, Jastrz\c{e}bski, and Tralle, where 
$(\g, \h) = (\mathfrak{e}_{7(7)},\, \mathfrak{f}_{4(4)} \oplus \ssl(2,\R))$ and $
(\mathfrak{e}_{8(8)},\, \mathfrak{f}_{4(4)} \oplus \mathfrak{g}_{2(2)})$.
\end{enumerate}

One can easily see that the example (ii) satisfies the assumptions of Theorem~\ref{thm:main-general} for $n=2$ 
(and therefore, is (P-surf)). 
On the other hand, $\SL(5,\R)$ does not admit a reductive subgroup locally isomorphic to $\OO(2,1) \times \OO(3,1)$, 
hence our method does not apply to the example (i). 
We have not checked yet if our method works for the two exceptional examples (iii). 
We thus ask: 

\begin{problem}
Determine if Okuda's homogeneous space \textup{(i)} and the two exceptional homogeneous spaces \textup{(iii)} 
are \textup{(P-surf)} or not. 
\end{problem}

See Remark~\ref{rmk:failed} for our failed attempt to show that Okuda's homogeneous space is (P-surf). 

We also make the following conjecture:

\begin{conjecture}\label{conj:other-G}
For any linear real simple Lie group $G$ with sufficiently high real rank, 
there exists a reductive subgroup $H$ of $G$ for which $G/H$ is \textup{(P-surf)} but not \textup{(P-ss)}. 
\end{conjecture}

\begin{remark}
The conjecture does not hold when $\rank_\R G \leqslant 3$. See Remark~\ref{rmk:forthcoming} below.
\end{remark}

We are inclined to believe that Theorem~\ref{thm:main-general} is enough to show this conjecture 
(even for any $G$ of real rank $\geqslant 5$, perhaps), although it may well be wrong. 
The case where $\rank_\R G = 4$ should be difficult: as we explained above, our method does not work for $G = \SL(5,\R)$. 

\subsection{On proper actions of other discrete subgroups}

Theorem~\ref{thm:main-example} showed that not only (P-surf) but even (P-cocH$^4$) is not enough to ensure (P-ss). 
We conjecture the following stronger result: 

\begin{conjecture}
For any $n \geqslant 2$, 
there is a homogeneous space $G/H$ of reductive type with the following two properties: 
\begin{itemize}
\item $G/H$ is \textup{(P-cocH$^n$)}, 
i.e.\ there exists a discrete subgroup of $G$ which is isomorphic to a cocompact lattice of $\OO(n,1)$ and acts properly on $G/H$. 
\item $G/H$ is not \textup{(P-ss)}, 
i.e.\ there does not exist a non-compact semisimple subgroup of $G$ which acts properly on $G/H$. 
\end{itemize}
\end{conjecture}

By the discussion in Section~\ref{sect:proof-cocHn}, this conjecture would follow once we could prove the following: 

\begin{conjecture}\label{conj:strict-domination}
For any $n \geqslant 2$, there exists a cocompact lattice $\Gamma$ of $\OO(n,1)$ and $N \geqslant n$ 
such that any neighbourhood of the standard embedding
\[
\Gamma \hookrightarrow \OO(n,1) \hookrightarrow \OO(N,1)
\]
in $\Hom(\Gamma, \OO(N,1))$ contains a representation strictly dominated by it. 
\end{conjecture}

It is known that, if $n \geqslant 5$, 
there does not exist a \emph{cocompact} right-angled reflection group in $\OO(n,1)$ 
(Potyagailo--Vinberg~\cite[\S 2]{PV05}), 
which is why Conjecture~\ref{conj:strict-domination} is open in these dimensions. 
It is, however, still possible to ask the following question: 

\begin{question}[{cf.\ Kapovich~\cite{Kap05}, Desgroseilliers--Haglund~\cite{DH13}}]
Which discrete group admits an embedding into a convex cocompact right-angled reflection group in $\OO(n,1)$, 
with $n$ arbitrarily large?
\end{question}

For instance, if we could find a cocompact lattice of $\OO(k,1)$ 
which embeds into a convex cocompact right-angled reflection group in $\OO(n,1)$ (for $n$ sufficiently large), 
then we would be able to conclude from Theorem~\ref{thm:right-angled} that (P-cocH$^k$) does not imply (P-ss). 

\begin{remark}
Many discrete groups are known to admit virtual embeddings into right-angled Coxeter groups. 
For instance, Bergeron--Haglund--Wise~\cite[Thm.~1.8]{BHW11} proved that 
every arithmetic lattice of $\OO(n,1)$ of simplest type virtually embeds into some right-angled Coxeter group. 
However, it seems difficult to control the signature of Tits--Vinberg representations 
of the target right-angled Coxeter group.
\end{remark}

\subsection{On proper actions of semisimple Lie groups}

As we mentioned in Remark~\ref{rmk:Okuda-symmetric}, Okuda proved that 
the condition $\ab^{-w_0} \not\subset W \ab_H$ implies (P-sl$_2$R) when $H$ is a symmetric subgroup of $G$ 
(and thus, the conditions (P-ss), (P-sl$_2$R), (P-surf), (P-free), and (P-nva) are all equivalent in this case). 
In particular, exotic phenomena like Theorem~\ref{thm:main-simplified} or Fact~\ref{fact:Okuda} 
never happens for such $G/H$. 
The first author, Jastrz\c{e}bski, Okuda, and Tralle proved in \cite[Prop.~4.1]{BJOT15} that 
the same holds whenever $\rank_\R H = 1$. 
It would be worthwhile to find other classes of $G/H$ having this property: 

\begin{problem}\label{prob:vna-equiv-to-sl2}
Find many more classes of homogeneous spaces of reductive type for which 
$\ab^{-w_0} \not\subset W \ab_H$ implies \textup{(P-sl$_2$R)}. 
\end{problem}

\begin{remark}\label{rmk:forthcoming}
In fact, we recently found two other classes of $G/H$ with the above property, 
which we would like to discuss in another paper: 
\begin{itemize}
\item $\rank_\R G \leqslant 3$. 
\item $H$ is locally a fixed-point subgroup of some inner automorphism of $G$. 
\end{itemize}
\end{remark}

Given that the condition (P-sl$_2$R) is now relatively well-understood, 
it is natural to search for homogeneous spaces that 
do admit a proper action of a closed subgroup locally isomorphic to $\SL(2,\R)$ 
but do not admit proper actions of `larger' reductive subgroups. 
In general, given two real semisimple Lie algebras $\g$ and $\lfrak$, 
it seems very difficult to classify all conjugacy classes of subalgebras of $\g$ isomorphic to $\lfrak$. 
However, when $\lfrak = \oo(3,1)$ or $\oo(2,2)$ (i.e.\ $\ssl(2,\C)$ or $\ssl(2,\R) \oplus \ssl(2,\R)$), 
we expect that the classification is, in principle, manageable: 
classify all subalgebras of $\g \otimes_\R \C$ isomorphic to $\ssl(2, \C) \oplus \ssl(2,\C)$, 
and then, check if each subalgebra in the list is `compatible' with the Satake diagram of $\g$ 
(in a similar sense as Okuda~\cite{Oku13}). 
We therefore ask the following: 

\begin{problem}\label{prob:method-o31}
Develop an efficient method to verify if a given homogeneous space of reductive type 
admits a proper action of reductive subgroup locally isomorphic to $\OO(3,1)$ (or equivalently, $\SL(2,\C)$), 
and classify all reductive symmetric spaces having this property. 
\end{problem}

Kazuki Kannaka told us that he accomplished the classification when $G/H$ is a classical symmetric space (private communication). 

\begin{question}\label{ques:symmetric-o31}
Does an analogue of Okuda's theorem~\cite{Oku13} hold for $\OO(3,1)$? 
More precisely, is there a reductive symmetric space $G/H$ with the following two properties? 
\begin{itemize}
\item $G/H$ is \textup{(P-cocH$^3$)}, i.e.\ 
there exists a discrete subgroup of $G$ which is isomorphic to a cocompact lattice of $\OO(3,1)$ and acts properly on $G/H$.
\item There does not exist a reductive subgroup which is locally isomorphic to $\OO(3,1)$ and acts properly on $G/H$.
\end{itemize}
\end{question}

It may be also interesting to study analogues of Problem~\ref{prob:method-o31} and Question~\ref{ques:symmetric-o31} for other semisimple Lie groups (e.g.\ $\OO(2,2)$ or $\OO(4,1)$). Okuda~\cite{Oku14} gave a necessary condition for a reductive symmetric space to admit proper actions of $\SL(3, \R)$ or $\SU(2, 1)$. 

\subsection{Other questions}

The discrete subgroup $\Gamma$ obtained by Theorem~\ref{thm:main-general} is contained in $L$. 
In particular, it is not Zariski dense in $G$ (unless $L = G$). 
It leads us to consider the following question: 

\begin{question}
In the setting of Theorem~\ref{thm:main-example}~(1), is there a Zariski dense discrete subgroup of $G$ 
which is isomorphic to a cocompact lattice of $\OO(n,1)$ and acts properly on $G/H$?
\end{question}

We have the following partial result, which unfortunately does not generalize to $(n, N) = (3,6), (4,8)$: 

\begin{proposition}
In the setting of Theorem~\ref{thm:main-example}~(1) with $(n, N) = (2, 3)$, there exists a Zariski-dense discrete surface subgroup of $G$ which acts properly on $G/H$. 
\end{proposition}

\begin{proof}
Take a discrete subgroup $\Gamma$ of $G$ which is isomorphic to a cocompact lattice of $\OO(2,1)$ and acts properly on $G/H$. We denote by $j \colon \Gamma \hookrightarrow G$ the embedding. 
Passing to finite-index subgroups if necessary, we may assume that $\Gamma$ is a surface group of genus $\geqslant 2 (\dim G)^2$. Then, there is a Zariski-dense representation of $\Gamma$ into $G$ which is arbitrarily close to $j$ by Kim--Pansu~\cite[Thm.~1]{KP15}. On the other hand, by Remark~\ref{rmk:sharp} and Kassel--Tholozan~\cite[Cor.~6.3~(2)]{KT24+}, any small deformation of $j$ is a sharp embedding into $G$ with respect to $H$. In particular, it acts properly on $G/H$. 
\end{proof}

Recall that the main motivation of this study is the relation between 
the existence of a `large' discontinuous group $\Gamma$ for $G/H$ and 
the existence of a `large' reductive subgroup $L$ acting properly on $G/H$. 
One standard way to measure the `largeness' of $\Gamma$ 
is its \textit{virtual cohomological dimension}, denoted as $\vcd(\Gamma)$. 
Recall the following well-known result: 

\begin{fact}[{Kobayashi~\cite[Cor.~5.5]{Kob89}}]
Let $G/H$ be a homogeneous space of reductive type, 
and let $\Gamma$ be a virtually torsion-free discrete subgroup of $G$ acting properly on $G/H$. 
Then, we have 
\[
\vcd(\Gamma) \leqslant \dim(G/K) - \dim(H/(K \cap H)),
\]
with equality if and only if $\Gamma$ acts cocompactly on $G/H$. 
\end{fact}

We ask the following variant of Kobayashi's standard quotient conjecture: 

\begin{question}
Assuming that a homogeneous space $G/H$ of reductive type is not \textup{(P-ss)}, 
can we say anything about the virtual cohomological dimension of a discontiuous group for $G/H$? 
Is it even possible for such $G/H$ to admit a cocompact discontinuous group, 
besides trivial cases (e.g.\ $G/H$ is compact or $H$ contains the semisimple part of $G$)?
\end{question}

Of course, if a non-compact homogeneous space $G/H$ of reductive type is not (P-ss) 
but admits a cocompact discontinuous group, 
then it is (an extremely strong instance of) a counterexample to Kobayashi's conjecture. 

In this paper, 
we used convex cocompact representations of discrete groups into $\OO(n,1)$ to construct exotic proper actions. 
It is natural to study if analogous classes of representations produce further examples: 

\begin{problem}
Construct new examples of exotic proper actions on homogeneous spaces of reductive type via 
one of the following classes of representations:
\begin{enumerate}[label = {\upshape (\roman*)}]
\item Convex cocompact representations into other simple Lie groups of real rank $1$. 
\item Tits--Vinberg representations of general right-angled Coxeter groups. 
In particular, the one-parameter family considered in Danciger--Gu\'eriaud--Kassel~\cite[\S 6]{DGK20}. 
\item Anosov representations (Labourie~\cite{Lab06}). 
In particular, surface group representations which belong to higher Teichm\"uller components, 
such as Hitchin representations. 
\end{enumerate}
\end{problem}

\begin{remark}\label{rmk:failed}
We record our failed attempt to construct 
a proper surface group action on Okuda's homogeneous space (Fact~\ref{fact:Okuda}) 
using Hitchin representations into $\SL(3,\R)$. 
Let $\Gamma$ be a surface group of genus $\geqslant 2$, and consider a discrete embedding 
\[
\Gamma \xrightarrow{(j, \rho)} \SL(2,\R) \times \SL(3,\R) \hookrightarrow \SL(5,\R),
\]
where $j$ is a discrete embedding and $\rho$ is a Hitchin representation. 
Our idea was to use Tholozan's inequality~\cite[Th.~3.9]{Tho17} and 
the fact that Hitchin representations avoid the walls of the Weyl chamber 
(Kapovich--Leeb--Porti~\cite[Th.~1.8]{KLP14+}, Gu\'eritaud--Guichard--Kassel--Wienhard~\cite[Th.~1.3]{GGKW17}). 
Unfortunately, they are not enough to ensure the existence of a proper surface group action on Okuda's homogeneous space.
\end{remark}

\begin{remark}
We here point out a minor mistake in \cite[Th.~3.9]{Tho17}. 
The original statement of the theorem is not correct as written, whilst the following slight modification is true: 

\begin{theorem-no-number}
Assume that $\Phi$ is not identically zero. 
For each $t > 0$, let $\rho_t \colon \Gamma \to \PSL(3,\R)$ be the Hitchin representation associated with $(J, t\Phi)$, 
and let $j \colon \Gamma \to \PSL(2,\R)$ be the Fuchsian representation uniformizing $(S, J)$. 
Then, there exist $A, B > 0$ such that, for a sufficiently large $t > 0$, 
\[
A t^{1/3} L_j \leqslant L_{\rho_t} \leqslant Bt^{1/3}L_j.
\]
\end{theorem-no-number}

The issue in the original proof of the upper-bound estimate is that the PDE in the last
paragraph of p.1399 is wrong: the correct PDE is
\[
\Delta \log(\sigma_t) = -1+\sigma_t^2 - 2t^3\sigma_t^{-4} |\Phi|^2
\]
(see e.g.\ Wang~\cite[Prop.~3.3]{Wan91} or Lofrin~\cite[Prop.~3.6.1]{Lof01}). 
Fortunately, since the leading term of the PDE is unchanged, one can prove the above modified estimate as follows. 
Arguing as in the first paragraph of p.1400, one obtains 
\[
\sigma_t(x_m)^6(1-\sigma_t(x_m)^{-2}) \leqslant 2 t^2 \| \Phi \|_\infty^2.
\]
Assume that $t \geqslant 2 \| \Phi \|_\infty^{-1}$ (recall that $\Phi$ is not identically zero). 
It follows from the estimate $\sigma_t \geqslant t^{1/3}|\Phi|^{1/3}$ in the third paragraph of p.1400 that
\[
\sigma_t(x_m)^6(1-\sigma_t(x_m)^{-2}) \geqslant (1-2^{-2/3})\sigma_t(x_m)^6.
\]
Thus, we have 
$\sigma_t \leqslant 2^{1/6}(1-2^{-2/3})^{1/6} \| \Phi \|_\infty^{1/3}$
on $S$, and therefore, 
\[
L^B_{\rho_t} \leqslant 2^{1/6} (1-2^{-2/3})^{1/6} \| \Phi \|^{1/3}_\infty L^P.
\]

The original proof of the lower-bound estimate implicitly assumes that $\Phi$ is not identically zero, too 
(see the fourth paragraph of p.1400). 
\end{remark}

\subsection*{Acknowledgements}

We would like to thank 
Kazuki Kannaka, Fanny Kassel, Toshiyuki Kobayashi, Takayuki Okuda, and Nicolas Tholozan 
for stimulating discussions on the topic of this paper.
We would like to thank Kazuki Kannaka also for telling us his unpublished result. 

This work started during the 
\emph{7th Tunisian--Japanese Conference 
``Geometric and Harmonic Analysis on Homogeneous Spaces and Applications'' in Honor of Professor Toshiyuki Kobayashi} 
in Monastir, Tunisia, 
and the two authors decided to write a joint paper during the ICTS program 
\emph{Zariski dense subgroups, number theory and geometric applications} (ICTS/ZDSG2024/01) 
in Bengaluru, India. 
We are grateful for the organizers of these conferences. 

YM was supported by JSPS KAKENHI Grant Numbers 19K14529 and 24K16922.

\noindent
Maciej Boche\'nski

\noindent
Faculty of Mathematics and Computer Science, 
University of Warmia and Mazury,
S{\l}oneczna 54, 10-710 Olsztyn, Poland

\noindent
E-mail address: \href{mabo@matman.uwm.edu.pl}{mabo@matman.uwm.edu.pl}

\bigskip

\noindent
Yosuke Morita

\noindent
Faculty of Mathematics, Kyushu University, 
744 Motooka, Nishi-ku, Fukuoka 819-0395, Japan

\noindent
E-mail address: \href{y-morita@math.kyushu-u.ac.jp}{y-morita@math.kyushu-u.ac.jp}

\end{document}